\documentclass[11pt,reqno]{amsart}
\usepackage{amsthm, amsfonts,  amsmath, amssymb}
\usepackage{comment}
\usepackage{url}
\usepackage[a4paper, top = 1.2in, bottom = 1.2in, left = 1.2in, right = 1.2in]{geometry}

\DeclareMathOperator{\cond}{Cond}

\newtheorem{thm}{Theorem}[section]
\newtheorem{lem}[thm]{Lemma}
\newtheorem{prop}[thm]{Proposition}
\theoremstyle{definition}

\newtheorem{comm}[thm]{Comment}

\newtheorem{ex}{Example}

\newtheorem{remark}[thm]{Remark}
\theoremstyle{plain}
\newtheorem{cor}[thm]{Corollary}
\numberwithin{equation}{section}

\newcommand{\C}{\mathbb{C}}

\newcommand{\N}{\mathbb{N}}
\newcommand{\Q}{\mathbb{Q}}

\newcommand{\Z}{\mathbb{Z}}

\newcommand{\SL}{\mathrm{SL}}

\newcommand{\lan}{\langle}
\newcommand{\ran}{\rangle}

\newcommand{\ord}{{\rm ord}}
\newcommand{\OO}{{\mathcal O}}
\newcommand{\Hup}{\mathbb{H}}

\newcommand{\mrm}[1]{\mathrm{#1}}
\newcommand{\mcal}[1]{\mathcal{#1}}

\newcommand{\kro}[2]{\left( \frac{#1}{#2} \right) }

\title[Fourier coefficients of half-integral weight modular forms]{A note on the Fourier coefficients of half-integral weight modular forms}

\author[N. Kumar]{Narasimha Kumar}
\email{narasimha.kumar@iith.ac.in}
\address{Department of Mathematics \\
Indian Institute of Technology Hyderabad \\
Ordnance Factory Estate \\
Yeddumailaram 502205 \\
Andhra Pradesh, INDIA.}

\author[S. Purkait]{Soma Purkait}
\email{Soma.Purkait@warwick.ac.uk}
\address{Mathematics Institute\\
              University of Warwick\\
             Coventry, CV4 7AL, United Kingdom.}


\begin{document}
\begin{abstract}
In this  note, we show that the algebraicity of the Fourier coefficients of  half-integral weight modular forms
can be determined by checking the algebraicity of the first few of them.
We also give a necessary and sufficient condition for a half-integral weight modular 
form to be in  Kohnen's +-subspace by considering only finitely many terms.
\end{abstract}
\maketitle
\section{Introduction}
In the theory of modular forms, it is of fundamental interest to understand the algebraicity of the Fourier coefficients 
of modular forms. 
It is well-known that a normalized Hecke eigenform of integral weight has algebraic Fourier coefficients, 
since the Hecke eigenvalues coincide with the Fourier coefficients. Moreover, there exists a number field
containing all these Fourier coefficients. However, for half-integral weight 
modular forms, the authors are not aware of any such results.

In this note, we show that if the Fourier coefficients of a half-integral weight modular form are algebraic 
up to Sturm's bound (for half-integral weight modular forms), which is specified in terms of the level and weight of the corresponding modular form, 
then so are all others (cf. Theorem~\ref{thm:alg} in the text). 

In his remarkable work \cite{Koh82}, Kohnen defines the $+$-subspace (new) and proves that it is isomorphic to a space of newforms of integral weight via certain 
linear combination of Shimura correspondences. In the last section, we give a necessary and sufficient condition for a half-integral weight modular 
form to be in Kohnen's $+$-subspace by checking the vanishing condition up to Sturm's bound. 

\section{Multiplication by Theta series}\label{Section:ThetaSeries}
Let $k>1$ be an odd integer. Let $f$ be a half-integral weight modular form of weight $k/2$, level $N$ with $4 \mid N$, and an even Dirichlet 
character $\chi$. In particular, $f \in M_{k/2}(\Gamma_1(N))$, therefore $f$ has a Fourier expansion of the form 
$f(z)=\sum_{n=0}^{\infty} a_f(n) q^n,$ where $q = e^{2 \pi iz}$. 
Let $\Theta(z)= \sum_{n=0}^{\infty} b_\Theta(n) q^{n^2}$ with $b_{\Theta}(0)=1$, $b_\Theta(n)=2$ for all $n \geq 1$.

In this section, we show that the Fourier coefficients of  $f$ are algebraic
if the Fourier coefficients of $f \Theta$ are algebraic.

By~\cite{Shimura}, we know that $\Theta \in M_{1/2}(4,\chi_{\mrm{triv}})$, where $\chi_{\mrm{triv}}$ stands 
for the trivial Dirichlet character. Define $$g :=f \Theta. $$
If $f \in M_{k/2}(N,\chi)$, then it is easy to see that $g \in M_{\frac{k+1}{2}}(N, \chi. \chi_{-1}^{\frac{k+1}{2}})$, 
where $\chi_{-1}$ is the non-trivial Dirichlet character modulo $4$.

Suppose $g(z)$ has the $q$-expansion given by 
$g(z) = \sum_{n=0}^{\infty} c(n) q^n$. Since $g$ is the product of
$f$ and $\Theta$, we see that for all $n \geq 0$, we have $$c(n) = \sum_{x+y^2=n} a_f(x) b_\Theta(y).$$

In general, the product of two eigenforms need not be an eigenform, in particular $f \Theta$ need not be an eigenform. 
Hence, one cannot conclude that the Fourier coefficients of $f \Theta$ generate a number field. 
However, if we assume that Fourier coefficients of $f \Theta$ generate a number field, then 
by a simple induction argument, we can show that the Fourier coefficients of $f$ also 
generate the same number field.

Instead of proving the above claim for $f$ and $\Theta$, we prove it for 
the product of two general half-integral weight modular forms.

\begin{prop}
\label{main-prop}
Let $f(z) = \sum_{n=n_0}^{\infty} a_n q^n \in M_{ k_1/2 }(N_1,\chi_1) $, 
$h(z) = \sum_{n=n_1}^{\infty} b_n q^n \in M_{k_2/2}(N_2,\chi_2) $ be two  modular forms
such that $a_{n_0} \not = 0$, $b_{n_1} \not =0$. 
Suppose that the Fourier coefficients of the product $fh$ belong to a number field $K$. 
If $a_i \in K$ for all $i \geq n_0$, then $b_i  \in K$ for all $i \geq n_1$.
\end{prop} 
\begin{proof}
We prove the proposition by induction. Suppose $$fh=\sum_{n=n_0+n_1}^{\infty} c_n q^n$$
is the $q$-expansion of $fh$. 

Look at the first non-zero coefficient of the product $fh$, i.e., $n_0+n_1$-th term of $fh$,
which is $c_{n_0+n_1} = a_{n_0} b_{n_1} \in K$. Since $a_{n_0} \not =0$, we get that $b_{n_1}$ is a non-zero element of $K$.

Now, look at the $n_0+n_1+1$-th term of the product $fh$,
which is $c_{n_0+n_1+1} = a_{n_0} b_{n_1+1} + a_{n_0+1} b_{n_1} \in K$. Since $a_{n_0}, a_{n_0+1}, b_{n_1} \in K$,
we get that $b_{n_1+1} \in K$.

Now, we assume that $b_{n_1}, b_{n_1+1}, \ldots, b_{n_1+r-1} \in K$ and show that $b_{n_1+r} \in K$.
We can write $$c_{n_0+n_1+r} = a_{n_0} b_{n_1+r} + a_{n_0+1} b_{n_1+r-1} + \ldots + a_{n_0+r} b_{n_1} \in K.$$
Since $b_{n_1},b_{n_1+1}, \ldots, b_{n_1+r-1} \in K$ and $a_{n_0},a_{n_0+1}, \ldots, a_{n_0+r} \in K$, 
we see that $b_{n_1+r}$ also belongs to $K$.
This proves the proposition.
\end{proof}

Observe that the above proof also works for any two general Fourier $q$-expansions. 
Now, coming back to the pair $(f,\Theta)$, we have the following:
\begin{cor}
If the coefficients $c_n(n \in \N)$ of $g=f\Theta$ belong to a number field $K$, then
the Fourier coefficients $a_f(n) \in K$ for all $n$.
\end{cor}
We shall illustrate the above proposition with an example.
\begin{ex}
Take $k=7$, $N=8$.  Using {\tt MAGMA}\cite{MAGMA}, we see that $S_{7/2}(8,\chi_{\mrm{triv}})$ is one-dimensional 
and is spanned by
\begin{equation*}
f =  q - 2q^2 - 4q^5 + 12q^6 - 3q^9 - 20q^{10} + O(q^{12}).
\end{equation*}
Now consider the integral weight modular form $g := f \Theta$. This is an element of $S_4(8, \chi_{\mrm{triv}})$,
which is of dimension $1$. Moreover, $S_4^{\mrm{new}}(8, \chi_{\mrm{triv}})$ is also of dimension $1$. By comparing
the Fourier coefficient of $q$ in $f\Theta$, we see that
$$S_4^{\mrm{new}}(8, \chi_{\mrm{triv}}) = \lan g \ran.$$ 
By Proposition~\ref{main-prop}, we see that the Fourier coefficients of $f$ generate a number field. 
\end{ex}

\begin{remark}
The same argument go through, if you replace $\Theta$ by $\Theta(\psi, 0,z)$ for $\psi$ an even Dirichlet character,
where $\Theta(\psi, 0,z)= \sum _{n=0}^{\infty} \psi(n) b_{\Theta}(n)$. 
 
Let $d$ be a positive integer and $V(d)$ be the usual shift operator. 
If the Fourier coefficients of $f(z) \cdot V(d)\Theta(\psi, 0,z)$ belongs to a number field $K$, then so 
are the Fourier coefficients of $f(z)$.
\end{remark}


\section{Main result}
In this section, we show that if the Fourier coefficients of a half-integral weight modular form are algebraic 
up to Sturm's bound (for half-integral weight modular forms), which is specified in terms of the level and 
weight of the corresponding modular form, then so are all others.

Let $k$, $N$ be positive integers with $k$ odd and $4 \mid N$. Let $\chi$ be an even Dirichlet character of modulus $N$. 
In his thesis Basmaji~\cite{Basmaji} gave an algorithm for computing a basis for the space of 
half-integral weight modular forms of level divisible by $16$. The main idea of the algorithm is to use theta series 
$\Theta = \sum_{n= -\infty}^{\infty} q^{n^2}$, 
$\Theta_1 =  \frac{\Theta - V(4)\Theta}{2}$ and the following embedding,
\[
\varphi : S_{k/2}(N, \chi) \rightarrow {S \times S},
\qquad \qquad
f \mapsto (f\Theta,f\Theta_1),
\]
where $S= S_{\frac{k+1}{2}}\left(N,\ \chi\cdot\chi_{-1}^\frac{k+1}{2}\right)$. This idea 
was later generalized by Steve Donnelly for levels divisible by $4$ by using different theta-multipliers~\cite{MAGMA}.

We will be requiring the following lemma, which is an analogue of Sturm's Theorem~\cite[page 276]{Sturm}
for half-integral weight modular forms.

\begin{lem}\label{lem:Sturm}
Let $f= \sum_{n=0}^{\infty} a_f(n) q^n \in M_{k/2}(N,\chi)$ with $k$ odd.
If $a_f(n)=0$ for $n\leq \frac{k}{24}[\SL_2(\Z):\Gamma_0(N)]$, then $f=0$.
\end{lem}
\begin{proof}
Set $B=\frac{k}{24}[\SL_2(\Z):\Gamma_0(N)]$.
Since $a_f(n)=0$ for $n \le B$, the Fourier expansion of $f$ at $\infty$ can be written as
$$ f = q^{B+1} (a_f(B+1) + a_f(B+2)q+ \cdots). $$

Let $s$ be the order of the Dirichlet character $\chi$. Then, it is easy to see that 
$f^{4s} \in M_{2ks}(\Gamma_0(N), \chi_\mrm{triv})$ is an integral weight modular form. Clearly, the
Fourier expansion of $f^{4s}$ at $\infty$ looks like
$$ f^{4s} = q^{4s(B+1)} \left(\sum_{n=0}^{\infty} c_n q^n \right),$$
where $c_n$ is in terms of $a_f(i)$ for $i \le n$.

Since the Fourier coefficients of $f^{4s}$ are zero up to $4sB=\frac{2ks}{12}[\SL_2(\Z):\Gamma_0(N)]$, 
by applying Sturm's theorem~\cite[page 276]{Sturm} to $f^{4s}$, we get that $f^{4s}=0$. This implies that
$f=0$, which proves the lemma.
\end{proof}

Our results are in the flavor of Sturm's work~\cite{Sturm} on determining modular forms by checking congruences 
up to a finite number. 
Set $B_k(N): = \frac{k}{24} \cdot [\SL_2(\Z):\Gamma_0(N)]$. 
\begin{thm}\label{thm:alg}
Let $f=\sum_{n=1}^{\infty} a_f(n) q^n \in S_{k/2}(N, \chi)$ be a non-zero half-integral weight modular form. 
Suppose $a_f(m)$ is algebraic for all $1 \leq m \leq B_k(N)$. Then all Fourier coefficients are algebraic. 
Moreover, there exists a number field $K_f$ such that $a_f(n) \in K_f$ for all $n$.
\end{thm}
\begin{proof}
By Basmaji's algorithm \cite{Basmaji}, we can construct a basis for the space of cusp forms 
$S_{k/2}(N, \chi)$ such that the basis elements have Fourier coefficients defined over the number field 
generated by $\chi$. 
Let $f_1, f_2, \ldots, f_r$ denote such a basis of $S_{k/2}(N, \chi)$. 
For each $i$, 
suppose $f_i := \sum_{n=1}^{\infty} a_i(n) q^n$ is the $q$-expansion. Write $ f = \sum_{i=1}^r \lambda_i f_i$  where 
$\lambda_i \in \C$. Hence, we have the following system of linear equations given by 
\begin{gather}\label{matrix:mat1}
\begin{bmatrix} 
a_1(1) &  a_2(1) &\cdots & a_r(1)\\ 
\vdots &  \vdots & \cdots & \vdots\\
a_1(m) &  a_2(m) &\cdots & a_r(m)\\
\vdots &  \vdots  & \cdots & \vdots\\
a_1(n) &  a_2(n) &\cdots & a_r(n)\\
\vdots &  \vdots  & \cdots & \vdots\\
\end{bmatrix}                 
\begin{bmatrix} 
\lambda_1 \\ \lambda_2 \\ \vdots \\\lambda_r
\end{bmatrix}
= 
\begin{bmatrix} 
a_f(1) \\ \vdots \\ a_f(m) \\ \vdots \\ a_f(n)  \\ \vdots 
\end{bmatrix}
\end{gather}
For simplicity, let $B$ denote $B_k(N)$.
Now from the above matrix, we consider the first $B$ rows to form the following $B \times r$-matrix, where $B \geq r$:
\begin{gather}
A:=
\begin{bmatrix} 
a_1(1) &  a_2(1) &\cdots & a_r(1)\\ 
\vdots &  \vdots & \cdots & \vdots\\
a_1(m) &  a_2(m) &\cdots & a_r(m)\\
\vdots &  \vdots  & \cdots & \vdots\\
a_1(B) &  a_2(B) &\cdots & a_r(B)\\
\end{bmatrix}                 
\end{gather}
We want to show that the rank of $A$ is $r$. Suppose not. 
Then, there exists $\alpha_1, \ldots, \alpha_r \in \C$, not all zero, 
such that 
\begin{equation}\label{matrix:mat3}
\sum_{i=1}^r \alpha_i a_i(j)  = 0, 
\end{equation}
for each $j= 1, \ldots, B$.
Now, consider the half-integral weight modular form $$h:= \alpha_1 f_1 +\alpha_2 f_2 + \ldots + \alpha_r f_r.$$
Now, by~\eqref{matrix:mat3}, the first $B$ coefficients of $h$ are zero. By Lemma~\ref{lem:Sturm}, 
we see that the modular form $h$ is identically zero. Therefore, we have 
$$ \alpha_1 f_1 +\alpha_2 f_2 + \ldots + \alpha_r f_r = 0.$$
Since $f_1, \ldots, f_r$ are linearly independent, we see that all $\alpha_i$'s have to be zero, which is a contradiction.
Therefore, the matrix $A$ has rank $r$. Let $C$ be the $r \times r$ submatrix of $A$ with full rank $r$.
Now, consider the following system of linear equations (formed from the~\eqref{matrix:mat1})
\begin{gather}
C \begin{bmatrix} 
\lambda_1 \\ \lambda_2 \\ \vdots \\\lambda_r
\end{bmatrix}
= 
\begin{bmatrix} 
a_f(i_1) \\ a_f(i_2) \\ \vdots \\ a_f(i_r)  
\end{bmatrix}
\end{gather}
for some distinct $1 \leq i_1 < i_2 < \cdots < i_r \leq B$.
Since $C$ is invertible, we see that $\lambda_i$'s can be expressed as an algebraic linear combination of $a_f(i_j)$ 
for $j=1,\ldots,r$. Therefore, $\lambda_i$'s are algebraic, hence all the Fourier coefficients of $f$ are algebraic.

Now, take $K_f$ to the number field generated by $\lambda_1, \ldots, \lambda_n$ and the values of $\chi$.
Since $f = \sum_{i=1}^n \lambda_i f_i$, we see that $a_f(n) \in K_f$ for all $n$.
\end{proof}
\begin{remark}
Using the same method one can see that the above theorem also holds for integral weight cusp forms for any level 
and nebentypus.
\end{remark}

One could also use the idea of multiplying the half-integral weight modular form $f$ of any level and nebentypus
by theta series $\Theta$ (as in Section~\ref{Section:ThetaSeries}) 
and argue in the integral weight case setting (cf.~\cite[Lemma 3.1]{KM12}). In this approach, one may have 
to check the algebraicity of the Fourier coefficients of $f\Theta$ up to  
$\frac{k+1}{24}[\SL_2(\Z):\Gamma_0(N)]$
to conclude that  all the other coefficients of $f\Theta$ are algebraic, and
hence the coefficients of $f$ as well. However, if we work completely in the half-integral 
weight setting, it is sufficient to check 
the algebraicity of the Fourier coefficients of $f$ up to $B:=\frac{k}{24}[\SL_2(\Z):\Gamma_0(N)]$ (in particular,
this implies that only the first $B$ coefficients of $f\Theta$ are algebraic).
Hence, we chose to work in the half-integral weight setting which gives a bit better bound.

We predict that for an eigenform $f$ such a bound for the algebraicity can be sharpened by using Waldspurger's 
results~\cite{Waldspurger}. We hope to come back to this in the future.

\section{When $N/4$ is odd and square-free}
Let $k$, $N$ be positive integers with $k \geq 3$ odd and $4 \mid N$. 
Let $\chi$ be an even quadratic Dirichlet character of modulus $N$. 

Let $F = \sum_{n=1}^{\infty} A(n)q^n$ be a newform of weight $k-1$, level $N/4$ odd and square-free
with trivial nebentypus. Let $f$ be a non-zero element of  $S_{k/2}(N,\chi,F)$ 
(for the definition, see \S\ref{sec:algbas}).
In this case, checking the algebraicity of the Fourier coefficients of $f$ 
becomes practically effective.

\subsection{An algebraic basis for $S_{k/2}(N,\chi,F)$:}\label{sec:algbas}
In this section, we shall recall some basic definitions and results.  

Let $S_{k/2}^\prime(N,\chi)$ be the orthogonal complement of 
the subspace of $S_{k/2}(N,\chi)$ spanned by single-variable theta series 
with respect to the Petersson inner product. 
Note that for $k \ge 5$, we have $S_{k/2}^\prime(N,\chi) = S_{k/2}(N,\chi)$.

Let $N^\prime=N/2$. For $M \mid N^\prime$ such that $\cond(\chi^2) \mid M$
and  $F \in S_{k-1}^{\mathrm{new}}(M,\chi^2)$,  Shimura defines
\[
S_{k/2}(N,\chi,F)=
\{
f \in S_{k/2}^{\prime}(N,\chi) : 
\text{$T_{p^2}(f)=\lambda_{p}^{F} f$ for almost all $p \nmid N$}\}; 
\]
here $T_{p} (F) =\lambda_{p}^{F} F$.
In~\cite[Corollary $5.2$]{SomaI} the second author gives an algorithm for computing these subspaces.
We will need the following proposition in the next section.
\begin{prop}
\label{Algebraic-basis}
Let $F$ be a newform in $S_{k-1}^{\mathrm{new}}(M,\chi^2)$ with level $M$ dividing $N^{\prime}$. 
Then there exists a basis of $S_{k/2}(N, \chi, F)$ defined over the number field generated 
by the Fourier coefficients of $F$ and $\chi$. 
\end{prop}

\begin{proof}
By Basmaji's algorithm~\cite{Basmaji}, one can construct a basis for the space of cusp forms $S_{k/2}(N, \chi)$ defined over the number field 
generated by $\chi$. Let $f_1, f_2, \dots f_r$ be such a basis and let $f_i = \sum_{n=1}^{\infty} a_i(n)q^n$.
Let $p_1 < p_2< \cdots <p_s$ be the primes chosen as in ~\cite[Corollary $5.2$]{SomaI}. Let 
$T_{p_j^2}f_i = \sum_{n=1}^{\infty} b_{i,j}(n)q^n$ for $1 \le i \le r$ and $1 \le j \le s$. 
By the same corollary
to construct a basis for $S_{k/2}(N, \chi, F)$, we need to determine a basis of the solution space for the simultaneous homogeneous 
system of linear equations given by 
\begin{gather*}
\begin{bmatrix} 
\vdots & \vdots & \hdots & \vdots\\
b_{1,1}(n) - \lambda_{p_1}^{F} a_1(n) &  b_{2,1}(n) - \lambda_{p_1}^{F} a_2(n) & \hdots & b_{r,1}(n) - \lambda_{p_1}^{F} a_r(n)\\ 
\vdots & \vdots & \hdots & \vdots\\
b_{1,s}(n) - \lambda_{p_s}^{F} a_1(n) &  b_{2,s}(n) - \lambda_{p_s}^{F} a_2(n) & \hdots & b_{r,s}(n) - \lambda_{p_s}^{F} a_r(n)\\
\vdots & \vdots & \hdots & \vdots\\
\end{bmatrix}                 
\begin{bmatrix} x_1 \\ x_2 \\ \vdots \\x_r\end{bmatrix}
= 
\begin{bmatrix} \vdots \\ 0  \\ \vdots \\ 0 \\ \vdots \end{bmatrix}.
\end{gather*}
Since the matrix of the system is defined over number field generated by $F$ and $\chi$, the proposition clearly follows. 
\end{proof}

\subsection{On determination of the algebraicity}
Let $F$ be a newform of weight $k-1$ and level $N/4$ odd and square-free 
with trivial nebentypus. 
Let $S_{k/2}^{+}(N,\chi)$ denote Kohnen's $+$-subspace of $S_{k/2}(N,\chi)$
consisting of modular forms $f = \sum_{n=1}^{\infty} a_f(n) q^n $ with
$a_f(n)=0$  for $n \equiv 2,(-1)^\frac{k+1}{2} \pmod 4$. Set $S^+_{k/2}(N,\chi,F)= S_{k/2}^{+}(N,\chi) \cap S_{k/2}(N,\chi,F)$.


By Waldspurger~\cite[Th\'{e}or\`{e}me 1]{Waldspurger}, we know that $S_{k/2}(N,\chi,F)$ is $2$-dimensional
as $N/4$ is odd and square-free.
By Proposition~\ref{Algebraic-basis}, we know that 
there exists a basis $f_1,f_2$ of  $S_{k/2}(N,\chi,F)$,
which are defined over the field generated by the Fourier coefficients of $F$ and  $\chi$. 
By~\cite[Theorem 2]{Koh82}, Kohnen's $+$-subspace $S^+_{k/2}(N,\chi,F)$ is $1$-dimensional. 
Suppose that $S^+_{k/2}(N,\chi,F)= \lan f_K \ran$.  

\begin{prop}
The Fourier coefficients of $f_K$ are algebraic up to a normalization.
\end{prop}
\begin{proof}
Let $f_K(z) = \sum_{n=1}^{\infty} a_K(n) q^n$.
Kohnen~\cite{Koh82} gives the following isomorphism 
\[ \mathcal{S}: S^{+,\ \mathrm{new} }_{k/2}(N,\chi) \overset\sim\longrightarrow S_{k-1}^{\mathrm{new}}(N/4),\]
where $\mathcal{S}$ is a certain linear combination of Shimura correspondences. Hence there exists a unique  
$g \in S^{+,\ \mathrm{new} }_{k/2}(N,\chi)$ such that $\mathcal{S}(g)=F$. Since the Shimura correspondence commutes 
with Hecke operators  for all primes $p$~\cite{SomaII}, we get 
\[\mathcal{S}(T_{p^2} g) = T_p (\mathcal{S} g) = T_p(F) = \lambda_p^F F = \mathcal{S}( \lambda_p^F g).\]
Since $\mathcal{S}$ is an isomorphism, $T_{p^2} g = \lambda_p^F g$ for all primes $p$ and thus $g$ equals $f_K$ up to a non-zero scalar. 
Consider the field $K_{f_K}=\Q( \mu_2,\ \{\lambda_p^F\}_{p \ne 2})$ where 
$\mu_2$ is the Hecke eigenvalue of $f_K$ under the Hecke operator $T^{+}_{4}$.
By~\cite{Koh82}, $\mu_2$ is algebraic and hence $K_{f_K}$ is a number field.

By~\cite{Koh92}, it follows that by multiplying $f_K$ with $a_K(|D_0|)^{-1}$ where $D_0$ is a suitable fundamental discriminant 
with $(-1)^ \frac{k-1}{2} D_0 > 0$ and $a_K(|D_0|) \not = 0$ one can assume that the Fourier coefficients 
of $f_K$ lie in $K_{f_K}$. Hence the Fourier coefficients of $f_K$ are algebraic up to a normalization.
\end{proof}
By the above proposition, we can now choose the basis element $f_1$ to be equal to $f_K$. 
By the argument as in Theorem~\ref{thm:alg},   the following result holds for $S_{k/2}(N, \chi, F)$. 
\begin{cor}\label{cor:m0n0}
Let $f=\sum_{n=1}^{\infty} a_f(n) q^n$ be an element of  $S_{k/2}(N,\chi,F)= \lan f_1,f_2 \ran$,
where $f_i (z) := \sum_{n=1}^{\infty} a_i(n) q^n$, with $a_i(*)$'s are algebraic.
Choose $m_0 \in \N$ with $m_0 \not\equiv 0, (-1)^{\frac{k-1}{2}} \pmod 4$ such that $a_1(m_0)=0$ and 
$a_2(m_0) \not =0$. Choose $n_0 \in \N$ such that $a_1(n_0) \not = 0$.

If  $a_f(m_0)$, $a_{f}(n_0)$ are algebraic, then $a_f(n)$ are algebraic for all $n \in \N$. 
Moreover, there exists a number field $K_f$ such that $a_f(n) \in K_f$ for all $n$.
\end{cor}
\begin{remark}
One can always find such an $m_0$, $n_0$ as in the above Corollary,
since  $f_2 \not \in S^+_{k/2}(N,\chi,F)$ and $f_1 \not = 0$.
 
\end{remark}

We illustrate the results in this section with a few examples.
\begin{ex}
\label{ex-1}
Let $F$ be the newform in $S_2^{\mathrm{new}}(91)$ given by the following Fourier expansion
\[F = q - 2q^2 + 2q^4 - 3q^5 - q^7 - 3q^9 + 6q^{10} - 6q^{11} + O(q^{12}).\]
We note that $F$ is the newform corresponding to the rank $1$ elliptic curve defined by $y^2 + y = x^3 + x$.
By~\cite[Corollary $5.2$]{SomaI}, the space $S_{3/2}(364,\chi_{\mrm{triv}},F)$ is generated by 
\begin{equation*}
\begin{split}
f_1 &=  q^3 - q^{12} + q^{35} - q^{40} + O(q^{50})\\ 
f_2 &= q^{10} + q^{12} + q^{13} - q^{14} + q^{17} + q^{26} - 3q^{38} - q^{40} - 2q^{42} - 2q^{48} + O(q^{50}).
\end{split}
\end{equation*}
Here $f_1$ belongs to Kohnen's $+$-subspace and where as $f_2$ does not. Thus to check whether a given 
eigenform in $S_{3/2}(364,\chi_{\mrm{triv}},F)$ has algebraic Fourier coefficients, we need to take $m_0=10$ and $n_0=3$ in Corollary~\ref{cor:m0n0}.
\end{ex}

\begin{ex}
\label{ex-2}
Let $G$ be the newform in $S_4^{\mathrm{new}}(13)$ given by 
\[G = q + bq^2 + (-3b + 4)q^3 + (b - 4)q^4 + (b - 2)q^5 + O(q^{6}),
    \]
where the minimal polynomial of $b$ is  $x^2 - x - 4$.
Let $\chi_{13}$ be the quadratic Dirichlet character given by the Kronecker symbol $\kro{13}{\cdot}$.
By~\cite[Corollary $5.2$]{SomaI}, the space $S_{5/2}(52,\chi_{13},G)$ is generated by 
\begin{equation*}
\begin{split}
f_1 &=  q + (b + 2)q^4 + 1/2(-5b - 8)q^5 + O(q^{6})\\ 
f_2 &= q^2 + 1/19(-2b + 14)q^3 + (-b + 1)q^4 + 1/19(18b + 26)q^5 + O(q^{6}).
\end{split}
\end{equation*}
Here $f_1$ belongs to Kohnen's $+$-subspace while $f_2$ does not. Thus for a given 
eigenform in $S_{5/2}(52,\chi_{13},F)$ to have algebraic Fourier coefficients, we just need to take $m_0=2$ and $n_0=1$ in 
Corollary~\ref{cor:m0n0}.
\end{ex}

\section{Sturm's bound for modular forms in Kohnen's +-subspace}
In this section, we will give a necessary and sufficient condition for a half-integral weight modular 
form to be in  Kohnen's +-subspace by considering only finitely many terms.

\begin{prop}
Suppose $f=\sum_{n=1}^{\infty} a_f(n) q^n \in S_{k/2}(N,\chi)$. Then, 
$f \in S^{+}_{k/2}(N,\chi)$ if and only if  
$a_f(n)=0$ for all $n \equiv 2,(-1)^\frac{k+1}{2} \pmod 4$
with $1 \leq n \leq B^{\prime}_k(N)$, where $B^{\prime}_k(N)= \frac{k}{24} \cdot [\SL_2(\Z):\Gamma_0(4N)]$.

\end{prop}
\begin{proof}
For simplicity, let $B$ denote $B^{\prime}_k(N)$. The necessary condition is clear, by definition. 


%
Suppose that $a_f(n) = 0$ for all $1 \leq n \leq B$ with $n \equiv 2, (-1)^{\frac{k+1}{2}} \pmod 4$.
Since $a_f(n) = 0$  for all $1 \leq n \leq B$ with $n \equiv (-1)^{\frac{k+1}{2}} \pmod 4$, 
Sturm's bound for the Fourier coefficients of half-integral weight forms in arithmetic 
progressions~\cite[Section 7]{SomaIII} implies that $a_f(n) = 0$ for all $n$ 
with $n \equiv (-1)^{\frac{k+1}{2}} \pmod 4$. 

Now, it is enough show that $a_f(n) = 0$ for all $1 \leq n \leq B$ with $n \equiv 2 \pmod 4$ 
implies that $a_f(n) = 0$ for all $n \in \N$ with $n \equiv 2 \pmod 4$.
Unfortunately we cannot apply the above result in this situation because $(2,4) \not = 1$.
However, we can get around to prove the required statement.

Applying $V$ and $U$ operators to $f$ (cf.~\cite{Shimura}), we obtain that
$$V(4) \circ U(4) f(z) = \sum_{n=1}^{\infty} a_f(4n) q^{4n} \in S_{k/2}(4N,\chi),$$
and
$$V(2) \circ U(2) f(z) = \sum_{n=1}^{\infty} a_f(2n) q^{2n} \in S_{k/2}(2N,\chi).$$
Take 
$$ g (z)= V(2) \circ U(2) f(z) - V(4) \circ U(4) f(z) = \sum_{n=1}^{\infty} 
          a_f(4n-2) q^{4n-2} \in S_{k/2}(4N,\chi).$$
Since $a_f(n)=0$ for all $1 \leq n \leq B$ with $n \equiv 2 \pmod 4$, 
we see that the first $B$ coefficients
of $g(z)$ are zero. By Lemma~\ref{lem:Sturm} 
we obtain that $g(z) = 0$. Hence $a_f(n) = 0$ for all $n \equiv 2 \pmod 4$. 
This proves the proposition.
\end{proof}

It follows from the above proposition that in Corollary~\ref{cor:m0n0},
we can choose $m_0$, $n_0 \le B^{\prime}_k(N)$.

\section{Acknowledgements}
We thank Prof. W. Kohnen for his encouragement and Prof. G. Wiese for some useful discussions. 
This work was initiated during the first author's visit to Hausdorff Research Institute for Mathematics, 
Bonn, and second author's visit to Max Planck Institute for Mathematics, Bonn, during Jan-Feb 2013. 
We thank these institutes for their hospitalities and also for providing perfect working conditions.

\bibliographystyle{amsalpha}

\end{document}